\documentclass[10pt]{amsart}
\newcommand{\zz}{{\mathbb Z}}
\newcommand{\nn}{{\mathbb N}}

\newcommand{\rr}{{\mathbb R}}
\newcommand{\beq}{\begin{eqnarray*}}
\newcommand{\feq}{\end{eqnarray*}}
\newcommand{\beqn}{\begin{eqnarray}}
\newcommand{\feqn}{\end{eqnarray}}

\newtheorem{theorem}{Theorem}[section]
\newtheorem{lemma}[theorem]{Lemma}
\newtheorem{corollary}[theorem]{Corollary}
\newtheorem{proposition}[theorem]{Proposition}

\theoremstyle{definition}

\theoremstyle{remark}

\numberwithin{equation}{section}

\begin{document}

\title{ON TIGHTNESS OF THE SKEW RANDOM WALKS}

\author{YOUNGSOO SEOL}
\address{Department of Mathematics, Iowa State University, Ames, Iowa 50010}
\email{yseol@iastate.edu}






\begin{abstract}
The primary purpose of this article is to prove a tightness of $\alpha$-skew random walks.
The tightness result implies, in particular, that the $\alpha$-skew Brownian motion can be constructed as
the scaling limit of such random walks. Our proof of tightness is based on a forth-order moment method.
\end{abstract}

\maketitle

\section{Introduction and statement of the main result}
\label{intro}
Skew Brownian motion was introduced by It\^{o} and Mckean \cite{itomckean1} to furnish
a construction of certain stochastic processes related to Feller's classification of
second order differential operators associated with diffusion processes
(see also Section~4.2 in \cite{itomckean2}).
For $\alpha \in (0,1),$ the $\alpha$-skew Brownian motion is defined as a one-dimensional Markov process
with the same transition mechanism as of the usual Brownian motion, with the only exception that the excursions away from
zero are assigned a positive sign with probability $\alpha$ and a negative sign
with probability $1-\alpha.$ The signs form an i.i.d. sequence and are chosen independently of the past history
of the process. If $\alpha=1/2,$ the process is the usual Brownian motion.
\par
Formally, the $\alpha$-skew random walk on $\zz$ starting at $0$ is defined as the birth-death Markov chain
$S^{(\alpha)}=\{\,S_{k}^{(\alpha)} ; k \geq 0\,\}$ with $S_{0}^{\alpha}=0$ and one-step transition probabilities given by
\beq
P\bigl(S_{k+1}^{(\alpha)}=m+1|S_{k}^{(\alpha)}=m\bigr)
&=&\left\{
\begin{array}{cl}
\alpha&\mbox{if}~m=0
\\
\frac{1}{2}&\mbox{otherwise}. \\
\end{array}
\right.
\\
\empty
\\
P\bigl(S_{k+1}^{(\alpha)}=m-1|S_{k}^{(\alpha)}=m\bigr)&=&
\left\{
\begin{array}{cl}
1-\alpha&\mbox{if}~m=0
\\
\frac{1}{2}&\mbox{otherwise}.
\end{array}
\right.
\feq
In the special case $\alpha=\frac{1}{2}$, $S^{(\frac{1}{2})}$ is a simple symmetric random walk on $\zz.$
Notice that when $\alpha\neq 1/2,$ the jumps (in general, increments) of the random walk are not independent.
\par
Harrison and Shepp \cite{harrisonshepp1} asserted (without proof) that the functional central
limit theorem for reflecting Brownian motion can be used to construct skew Brownian motion as the
limiting process of a suitably modified symmetric random walk on the integer lattice.
This result has served as a foundation for numerical algorithms tracking moving particle in a
highly heterogeneous porous media; see, for instance, \cite{hoteit1,lejay1,Appuhamillage2,saxton1}.
In \cite{lejay1} it was suggested that tightness could be obtained based on second moments, however this is
not possible even in the case of simple symmetric random walk. The lack of statistical independence of the
increments make a fourth moment proof all the more challenging.
Although proofs of FCLTs in more general frameworks have subsequently been obtained by other methods, e.g., by Skorokhod
embedding in \cite{Cherny1}, a self-contained simple proof of tightness for simple skew random walk has not been available
in the literature.
\par
The main goal of this paper is to prove the following result.
Let $C(\rr_+,\rr)$ be the space of continuous functions from $\rr_+=[0,\infty)$ into $\rr,$
equipped with the topology of uniform convergence on compact sets.
For $n\in \nn$, let $X_n^{(\alpha)} \in C(\rr_+,\rr)$ denote the following linear interpolation of $S^{(\alpha)}_{[nt]}$:
\beq
X_n^{(\alpha)}(t) = \frac{1}{\sqrt{n}}\Bigl(S^{(\alpha)}_{[nt]} + (nt-[nt]) \cdot S^{(\alpha)}_{[nt]+1}\Bigr).
\feq
Here and henceforth $[x]$ denotes the integer part of a real number $x.$
\begin{theorem}
\label{main-thm}
For any $\alpha \in (0,1),$ there exists a constant $C_\alpha>0,$ such that the inequality
\beq
E\bigl|X_n^{(\alpha)}(t)-X_n^{(\alpha)}(s)\bigr|^{4} \leq C_\alpha |s-t|^2,
\feq
holds uniformly for all $s,t>0,$ and $n\in \nn.$
\end{theorem}
The results stated above implies (see, for instance, \cite[p.~98]{bhattwaymire}):
\begin{corollary}
The family of processes $X^{(\alpha)}_n,$ $n\in \nn,$ is tight in $C(\rr_+,\rr).$
\end{corollary}
The rest of the paper is organized as follows. In Section~\ref{prelim} we define $\alpha$-skew random walks and review
some of their basic properties. Theorem~\ref{main-thm} and the tightness property for the skew random walks are
proved in Section~\ref{proofs}.
\section{Proof of Theorem~\ref{main-thm}}
\label{proofs}
In this section we complete the proof of our main result, Theorem~\ref{main-thm}.
\label{prelim}
In what follows we will use $S$ to denote the simple symmetric random walk $S^{(\frac{1}{2})}.$
The following observations can be found in \cite{harrisonshepp1}.
\begin{proposition}
\label{prop:samdis}
\item [(a)] $\left|S^{(\alpha)} \right|$ has the same distribution as $\left|S \right|$ on $\zz_+=\{\,0,1,2,\dots\,\}$.
That is, $\left|S^{(\alpha)} \right|$ is a simple symmetric random walk on $\zz_+,$ reflected at 0.
\item [(b)] The processes $-S^{(\alpha)}$ and $S^{(1-\alpha)}$ have the same distribution.
\end{proposition}
The next statement describes $n$-step transition probabilities of the skew random walks
by relating them to those of $S$ (see, for instance, \cite[p.~436]{lejay1}).
\begin{proposition}
\label{prop23}
For $m \in Z$, $k >0$
\beqn
&&P\bigl(S_k^{(\alpha)}=m\bigr)
=\left\{
\begin{array}{lll}
\alpha\cdot P\bigl(|S_k|=m\bigr)&\rm{if}& m>0
\\
(1-\alpha)\cdot P\bigl(|S_k|=-m\bigr)&\rm{if}& m<0
\\
P\bigl(|S_k^{(\alpha)}|=0\bigr)=P\bigl(|S_k|=0\bigr)& \rm{if}& m=0
\end{array}
\right.\nonumber
\feqn
\end{proposition}
The following observation is evident from the explicit form of
the distribution function of $S_k^{(\alpha)},$ given in Proposition~\ref{prop23}.
\begin{proposition}
\label{lemma1}
With probability one,
\beq
&&E\bigl(S_{j+1}^{(\alpha)}-S_{j}^{(\alpha)}\bigl|S_{j}^{(\alpha)}\bigr)=(2\alpha-1)\textbf{1}_{\{S_{j}^{(\alpha)}=0\}}
\feq
and
\beq
E\bigl[\bigl(S_{i+1}^{(\alpha)}-S_i^{(\alpha)}\bigr)^2\bigl|S_i^{(\alpha)}\bigr]=1,
\feq
\end{proposition}
$\mbox{}$
\\
To show the result of Theorem~\ref{main-thm}, we will need a corollary to Karamata's Tauberian theorem,
which we are going now to state. For a measure $\mu$ on $[0,\infty),$ denote by
$\widehat{\mu}(\lambda):=\int_{0}^{\infty}e^{-{\lambda}x}\,\mu(dx)$ the Laplace transform of $\mu.$
The transform is well-defined for $\lambda \in (c,\infty),$ where $c>0$ is
a non-negative constant, possibly $+\infty.$  If $\mu$ and $\nu$ are measures on $[0,\infty)$ such that $\widehat{\mu}(\lambda)$
and $\widehat{\nu}(\lambda)$ both exist for all $\lambda >0$, then the convolution
$\gamma={\mu} \ast {\nu}$ has the Laplace transform $\widehat{\gamma}(\lambda)=
\widehat{\mu}(\lambda)\widehat{\nu}(\lambda)$ for $\lambda >0.$
If $\mu$ is a discrete measure concentrated on $\zz_+,$ one can identify $\mu$ with a
sequence $\mu_n$ of its values on $n\in\zz_+.$ For such discrete measures, we have
(see, e.g., Corollary~8.10 in \cite[p.~118]{bhattway}).
\begin{proposition}
\label{consequence}
Let $\widetilde{\mu}(t)=\sum_{n=0}^{\infty}\mu_{n}t^{n}, 0 \leq t <1$,
where $\{\,\mu_{n}\,\}_{n=0}^{\infty}$ is a sequence of non-negative numbers. For $L$ slowly varying at
infinity and $0 \leq \theta < \infty$ one has
\beq
&&\widetilde{\mu}(t)\sim (1-t)^{-\theta}L\Bigl(\frac{1}{1-t}\Bigr)~\mbox{\rm as}~t \uparrow 1
\feq
if and only if
\beq
&&{\sum_{j=0}^{n}\mu_{j}}\sim {\frac{1}{\Gamma(\theta)}}n^{\theta}L(n)~\mbox{\rm as}~ n \to {\infty}.
\feq
\end{proposition}
Here and henceforth, $a_n\sim b_n$ for two sequence of real numbers $\{a_n\}_{n\in \nn}$ and $\{b_n\}_{n\in \nn}$
means $\lim_{n\to\infty} a_n/b_n=1.$
\par
We are now in a position to prove the following key proposition. Define a sequence $\{q(k)\}_{k\in\zz_+}$ as follows
\beq
g(k) =
\left\{
\begin{array}{lll}
0 &\mbox{\rm if}&~ k\in\nn~\mbox{\rm is odd} \\
\\
\binom{2i}{i}2^{-2i}&\mbox{\rm if}&~ k=2i\in\nn~\mbox{\rm is even}.
\end{array}
\right.
\feq
Note that in view of Proposition~\ref{prop23},
\beq
g(k)=P\bigl(S_k=0\bigr)=P\bigl(|S_k|=0\bigr)=P\bigl(|S_k^{(\alpha)}|=0\bigr)=P\bigl(S_k^{(\alpha)}=0\bigr).
\feq
\begin{proposition}
\label{key-prop}
\item[(a)] If $\mu(j)=g \ast g(j)$ then $\sum_{j=0}^m\mu(j) \sim m$
\item[(b)] If $\nu(j)=g \ast g \ast g \ast g(j)$ then $\sum_{j=0}^m\nu(j) \sim m^{2}$.
\end{proposition}
\begin{proof}
For $t\in (0,1),$ let $\widetilde g(t)=\sum_{k=0}^{\infty}g(k)t^k.$ Notice that $\widetilde g(t)$ is well-defined
since $g(k)=P\bigl(S_k=0\bigr)<1$  for $k\geq 0.$ Since $g(2j)=\binom{2j}{j}2^{-2j}=(-1)^j\binom{-\frac{1}{2}}{j},$ we have
\beq
\widetilde g(t)&=&\sum_{k=0}^{\infty}g(k)t^k=\sum_{j=0}^{\infty}\binom{2j}{j}2^{-2j}t^{2j}=
\sum_{j=0}^{\infty}(-1)^{j}\binom{-\frac{1}{2}}{j}t^{2j}\\
&=&\sum_{j=0}^{\infty}\binom{-\frac{1}{2}}{j}(-t^{2})^{j}=(1-t^{2})^{-\frac{1}{2}}.
\feq
Notice that, using the notation of Proposition~\ref{consequence}, $\widetilde g(t)=\widehat g(\lambda)$ if
$t=e^{-\lambda}.$  Therefore, $\widetilde \mu(t)= {\widetilde g}^2(t)=(1-t^2)^{-1}$
while $\widetilde \nu(t)={\widetilde g}^4(t)=(1-t^2)^{-2}.$
Thus claims (a) and (b) of the proposition follow from Proposition~\ref{consequence} applied, respectively,
with $\theta=1,$ $L=1$ for $\mu$ and with $\theta=2,$ $L=1$ for $\nu.$
\end{proof}
The last technical lemma we need is the following claim.
\begin{lemma}
For integers $0<i_1<i_2<i_3<i_4$ define
\beq
A(i_1,i_2,i_3):=E(S_{{i_{3}}+1}^{(\alpha)}-S_{i_{3}}^{(\alpha)})^{2}
(S_{i_{2}+1}^{(\alpha)}-S_{i_{2}}^{(\alpha)})
(S_{i_{1}+1}^{(\alpha)}-S_{i_{1}}^{(\alpha)}),
\feq
and
\beq
B(i_1,i_2,i_3,i_4):=E(S_{{i_4}+1}^{(\alpha)}-S_{i_{4}}^{(\alpha)})
(S_{{i_3}+1}^{(\alpha)}-S_{i_3}^{(\alpha)})
(S_{i_2+1}^{(\alpha)}-S_{i_2}^{(\alpha)})
(S_{i_1+1}^{(\alpha)}-S_{i_1}^{(\alpha)}).
\feq
Then there is a constant $C>0$ such that
\beq
\sum_{1 \leq i_1 < i_2 < i_3 \leq k-j}
A(i_1,i_2,i_3) \leq C |k-j|^2,
\feq
and
\beq
\sum_{1 \leq i_{1} < i_{2} < i_{3} < i_{4}\leq k-j}
B(i_1,i_2,i_3,i_4)\leq C |k-j|^2.
\feq
\end{lemma}
\begin{proof}
Using Proposition~\ref{lemma1}, the Markov property, and the fact the excursions of $S^{(\alpha)}$ away from
zero are the same as excursions of the simple symmetric random walk $S,$ we obtain
\beq
&&A(i_1,i_2,i_3)
\\
&&
\qquad
=
E(S_{{i_{3}}+1}^{(\alpha)}-S_{i_{3}}^{(\alpha)})^{2}
(S_{i_{2}+1}^{(\alpha)}-S_{i_{2}}^{(\alpha)})
(S_{i_{1}+1}^{(\alpha)}-S_{i_{1}}^{(\alpha)}\textbf{1}_{\{S_{i_1}^{(\alpha)}=0\}}\textbf{1}_{\{S_{i_2}^{(\alpha)}=0\}}\bigr)
\\
&&
\qquad
=
P\bigl(S_{i_1}=0\bigr)\cdot (2\alpha-1)\cdot P\bigl(S_{i_2}=0|S_{i_1}=0\bigr)\cdot (2\alpha-1)
\\
&&
\qquad
=
(2\alpha-1)^2 g(i_1)g(i_2-i_1).
\feq
Therefore,
\beq
\sum_{1 \leq i_{1} < i_{2} < i_{3} <\leq k-j}
A(i_1,i_2,i_3)\leq\sum_{i_3=0}^{[k-j]}\sum_{i_2=0}^{i_3-1}\sum_{i_1=0}^{i_2-1} g(i_{2}-i_{1})g(i_{1}).
\feq
Using Proposition~\ref{key-prop}, we obtain
\beq
&&\sum_{i_{3}=0}^{[k-j]}\sum_{i_{2}=0}^{i_{3}-1}\sum_{i_{1}=0}
^{i_{2}-1}g(i_{2}-i_{1})g(i_{1})=\sum_{i_{3}=0}^{[k-j]}\sum_{i_{2}=0}^{i_{3}-1}
g\ast g(i_{2})
\leq \sum_{i_{3}=0}^{[k-j]}\sum_{i_{2}=0}^{[k-j]}
g\ast g(i_{2})
\\
&&
\qquad
\leq C_1 \left|k-j\right|^2,
\feq
for some constant $C_1>0$ and any $k,j\in\nn.$
\par
Similarly,
\beq
&&B(i_1,i_2,i_3,i_4)=
(2\alpha-1)^4\cdot P\bigl(S_{i_1}=0\bigr) \cdot \prod_{a=1}^3 P\bigl(S_{i_a+1}=0|S_{i_a}=0\bigr)
\\
&&
\qquad
=
(2\alpha-1)^4 g(i_1)g(i_2-i_1)g(i_3-i_2)g(i_4-i_3).
\feq
Hence, using again Proposition~\ref{key-prop},
\beq
\sum_{0 \leq i_1 < i_2 < i_3 < i_4} B(i_1,i_2,i_3,i_4) \leq \sum_{i_4=0}^{[k-j]}g \ast g \ast g \ast g(i_{4})
\leq C_2 |k-j|^2,
\feq
for some constant $C_2>0$ and any $k,j\in\nn.$
\end{proof}
$\mbox{}$
\\
We are now in a position to complete the proof of our main result.
\\
$\mbox{}$
\\
{\em Completion of the proof of Theorem~\ref{main-thm}.}
\\
First consider the case where $s=\frac{j}{n}<\frac{k}{n}=t$
are grid points. Then
\beq
&&
E\Bigl|\frac{S_{[nt]}^{(\alpha)}}{\sqrt{n}}-\frac{S_{[ns]}^{(\alpha)}}
{\sqrt{n}}\Bigr|^4=\frac{1}{n^2}E\Bigl|S_k^{(\alpha)}-S_j^{(\alpha)}\Bigr|^4
=\frac{1}{n^{2}}E\Bigl|\sum_{i=j}^{k-1}\bigl(S_{i+1}^{(\alpha)}-S_i^{(\alpha)}\bigl)\Bigr|^{4}
\\
&&
\quad
=\frac{1}{n^2}\sum_{i=j}^{k-1}E\bigl(S_{i+1}^{(\alpha)}-S_i^{(\alpha)}\bigr)^4
+\frac{1}{n^2}\sum_{ i_{1} < i_{2} \leq k-j} E\bigl(S_{i_{1}+1}^{(\alpha)}-S_{i_{1}}^{(\alpha)}\bigr)^2
\bigl(S_{i_2+1}^{(\alpha)}-S_{i_2}^{(\alpha)}\bigr)^2
\\
&&
\qquad
+\frac{1}{n^{2}}\sum_{ i_{1} < i_{2} < i_{3} \leq k-j}
E\bigl(S_{i_{3}+1}^{(\alpha)}-S_{i_{3}}^{(\alpha)})^{2}
\bigl(S_{i_{2}+1}^{(\alpha)}-S_{i_{2}}^{(\alpha)}\bigr)\bigl(S_{i_{1}+1}^{(\alpha)}-S_{i_{1}}^{(\alpha)}\bigr)
\\
&&
\qquad
+\frac{1}{n^{2}}\sum_{ i_{1} < i_{2} < i_{3} < i_{4} \leq k-j}
\prod_{a=1}^4 E\bigl(S_{i_a+1}^{(\alpha)}-S_{i_a}^{(\alpha)}\bigr)
\\
&&
\quad
\leq
\frac{1}{n^{2}}\sum_{i=j}^{k-1}{1}+\frac{1}{n^2}\binom{k-j}{2}\binom{k-j}{2}+\frac{1}{n^2}C_1
\bigl|k-j\bigr|^2+\frac{1}{n^{2}}C_2\bigl|k-j\bigr|^{2}
\\
&&
\quad
\leq C_3\bigl|t-s\bigr|^2,
\feq
for a large enough constant $C_3>0.$
\par
To conclude the proof of Theorem~\ref{main-thm},
it remains to observe that for non-grid points $s$ and $t$ one can use an approximation by neighbor grid points. 
In fact, the approximation argument given in \cite[pp.~100-101]{bhattwaymire} for regular random walks goes through verbatim.
\qed

\section*{Acknowledgment}
I would like to thank Professor Edward C. Waymire for suggesting this problem and for helpful comments.
I also want to thank Professor Alexander Roitershtein for helpful suggestions and corrections.
\bibliographystyle{amsplain}

\end{document}